\documentclass[11pt]{article}
\usepackage{amssymb}
\usepackage{amsmath}
\usepackage{graphicx}
\usepackage{setspace}
\begin{document}

\newtheorem{theorem}{Theorem}[section]
\newtheorem{tha}{Theorem}
\newtheorem{conjecture}[theorem]{Conjecture}
\newtheorem{corollary}[theorem]{Corollary}
\newtheorem{lemma}[theorem]{Lemma}
\newtheorem{claim}[theorem]{Claim}
\newtheorem{proposition}[theorem]{Proposition}
\newtheorem{construction}[theorem]{Construction}
\newtheorem{definition}[theorem]{Definition}
\newtheorem{question}[theorem]{Question}
\newtheorem{problem}[theorem]{Problem}
\newtheorem{remark}[theorem]{Remark}
\newtheorem{observation}[theorem]{Observation}

\newcommand{\ex}{{\mathrm{ex}}}

\newcommand{\EX}{{\mathrm{EX}}}

\newcommand{\AR}{{\mathrm{AR}}}

\def\endproofbox{\hskip 1.3em\hfill\rule{6pt}{6pt}}
\newenvironment{proof}%
{%
\noindent{\it Proof.}
}%
{%
 \quad\hfill\endproofbox\vspace*{2ex}
}
\def\qed{\hskip 1.3em\hfill\rule{6pt}{6pt}}
\def\ce#1{\lceil #1 \rceil}
\def\fl#1{\lfloor #1 \rfloor}
\def\lr{\longrightarrow}
\def\e{\varepsilon}
\def\ex{{\rm\bf ex}}
\def\cB{{\cal B}}
    \def\cD{{\cal D}}
\def\cF{{\cal F}}
\def\cG{{\cal G}}
\def\cH{{\cal H}}
\def\ck{{\cal K}}
\def\cI{{\cal I}}
\def\cJ{{\cal J}}
\def\cL{{\cal L}}
\def\cM{{\cal M}}
\def\cP{{\cal P}}
\def\cQ{{\cal Q}}
\def\cS{{\cal S}}
\def\imp{\Longrightarrow}
\def\1e{\frac{1}{\e}\log \frac{1}{\e}}
\def\ne{n^{\e}}
\def\rad{ {\rm \, rad}}
\def\equ{\Longleftrightarrow}
\def\pkl{\mathbb{P}^{(k)}_\ell}
\def\podd{\mathbb{P}^{(k)}_{2t+1}}
\def\peven{\mathbb{P}^{(k)}_{2t+2}}
\def\TT{{\mathbb T}}
\def\bbT{{\mathbb T}}
\def\cE{{\mathcal E}}

\def\wt{\widetilde}
\def\wh{\widehat}
\voffset=-0.5in

\pagestyle{myheadings}
\markright{{\small \sc F\"uredi, Jiang, Seiver:}
  {\it\small Tur\'an numbers of $k$-uniform linear paths}}

\title{\huge\bf Exact solution of the hypergraph Tur\'an problem for $k$-uniform linear paths}

\author{
Zolt\'an F\"uredi\thanks{Dept. of Mathematics, University of Illinois, Urbana, IL 61801, USA. E-mail:
z-furedi@illinois.edu
\newline
Research supported in part by the Hungarian National Science Foundation
 OTKA, by the National Science Foundation under grant NFS DMS 09-01276, and by a European Research Council Advanced Investigators Grant 267195.
}\quad\quad
Tao Jiang\thanks{Dept. of Mathematics, Miami University, Oxford,
OH 45056, USA. E-mail: jiangt@muohio.edu. } \quad \quad Robert Seiver\thanks{Dept. of Mathematics,
Miami University, Oxford, OH 45056, USA. E-mail:
seiverrs@muohio.edu.\newline
\newline \quad
    %{\rm\small {\jobname}.tex,}\newline ${}$ \hfill
version of June 28, 2011.
      %%%
} }

\date{}

\maketitle
\begin{abstract}
A $k$-uniform  linear path of length $\ell$, denoted by  $\pkl$, is a family of $k$-sets $\{F_1,\ldots, F_\ell\}$ such that
$|F_i\cap F_{i+1}|=1$ for each $i$ and $F_i\cap F_j=\emptyset$ whenever $|i-j|>1$.
 Given a $k$-uniform hypergraph $H$ and a positive integer $n$,
the {\it $k$-uniform hypergraph Tur\'an number} of $H$, denoted by
$\ex_k(n,H)$, is the maximum number of edges in a $k$-uniform
hypergraph $\cF$ on $n$ vertices that does not contain $H$ as a
subhypergraph. With an intensive use of the delta-system method,
we determine $\ex_k(n,P^{(k)}_\ell)$ exactly for all fixed
$\ell\geq 1, k\geq 4$, and sufficiently large $n$.
 We show that
$$\ex_k(n,\podd)={n-1\choose k-1}+{n-2\choose k-1}+\ldots+{n-t\choose k-1}.$$
The only extremal family consists of all
the $k$-sets in $[n]$ that meet some fixed set of $t$ vertices.
We also show that
$$\ex(n,\peven)={n-1\choose k-1}+{n-2\choose k-1}+\ldots+{n-t\choose k-1}+{n-t-2\choose k-2},$$
and describe the unique extremal family. Stability results on these bounds and some related results are also established.
\end{abstract}

%%%%%%%%%%%%%%%%%%%%%%%%%%%%%%%%%%%%%%%%%%%%%%%%%%%%%%%%%%%%%%%%%%%%%
\section{Introduction}
As usual, a hypergraph $\cF=(V,E)$ consists of a set $V$ of vertices and a set $E$ of edges, where each edge
is a subset of $V$. We call edges of $\cF$ {\it members} of $\cF$.
If each member of $\cF$ is a $k$-subset of $V$, we say that $\cF$ is a {\it $k$-uniform hypergraph} or a {\it $k$-uniform set system}.  If $|V|=n$, it is often convenient to just let $V=[n]=\{1,\ldots, n\}$.
For convenience, we  write $\cF\subseteq {[n]\choose k}$ to indicate that $\cF$ is a $k$-uniform hypergraph on
 vertex set $[n]$.
There is a long history in the study of extremal problems concerning hypergraphs.
Early well-known results include %%% for instance
   the Erd\H{o}s-Ko-Rado theorem that
says that for all $n\geq 2k$ the maximum size of a $k$-uniform family on $n$ vertices in which every two members intersect
is ${n-1\choose k-1}$, with equality achieved by taking all the subsets of $[n]$ containing a fixed element.   Given a family $\cH$ of hypergraphs, the {\it $k$-uniform hypergraph Tur\'an number} of $\cH$, denoted by
$\ex_k(n,\cH)$, is the maximum number of edges in a $k$-uniform hypergraph $\cF$ on $n$ vertices that does not contain a member of $\cH$ as a subhypergraph.
   An $\cH$-free family $\cF\subseteq {[n]\choose k}$ is called {\it extremal}  if $|\cF|=\ex_k(n,\cH)$.
If $\cH$ consists of a single hypergraph $H$, we write $\ex_k(n,H)$ for
$\ex_k(n,\{H\})$.
If we let $M^{(k)}_2$ denote the $k$-uniform hypergraph consisting of two disjoint $k$-sets, then the Erd\H{o}s-Ko-Rado theorem says $\ex_k(n,M^{(k)}_2)={n-1\choose k-1}$ for all $n\geq 2k$.
More generally, %%% let $M^{(k)}_t$ denote the $k$-uniform hypergpraph consisting of $t$ pairwise disjoint $k$-sets,
 Erd\H{o}s showed

\begin{theorem} {\bf (Erd\H{o}s~\cite{erdos-matching})} \label{matching}
Let $k,t$ be positive integers.  %%%, where $k\geq 2$ and $t\geq1$.
   There exists a number $n(k,t)$
such that  for all integers $n>n(k,t)$, if $\cF\subseteq
{[n]\choose k}$ contains no $t+1$ pairwise disjoint members then
$$|\cF|\leq {n\choose k}-{n-t\choose k}.$$
Furthermore, %%% equality holds if
  the only extremal family $\cF$ consists of all the $k$-sets of $[n]$ meeting some
fixed set $S$ of $t$ elements of $[n]$.
\end{theorem}
Surveys on Tur\'an problems of graphs and hypergraphs can be found in~\cite{ZF} and  ~\cite{keevash}.
Hypergraph Tur\'an problems are notoriously difficult.  The aymptotics are determined for very
few hypergraphs and exact results are particularly rare. Most exact results
concern specific hypergraphs on a small number of vertices (and often for fixed small values of $k$).
For example, the exact value of
$\ex_k(n,H)$ is determined (for large $n$) for the Fano plane, expanded triangle, $4$-books with $2$
pages, $4$-books with $3$ pages, $4$-books with $4$ pages, some $3$-graphs with independent neighborhoods,
extended complete graphs, generalized fans, and a couple of others (see~\cite{keevash} for details and references).
By comparison, our results in this paper establish the exact value  for every hypergraph in an infinite family (and for all $k\geq 4$).
In this regard, the exact result on extended complete graphs~\cite{oleg} (refining~\cite{dhruv}) is similar in nature.
However, the hypergraphs $H$ we consider are much more  sparse and more ``spread out''. So, our result may be viewed
the first of its kind.

\section{The Hypergraph problem for paths and main results}

In this paper, we focus on the hypergraph problem for paths. As explained at the end of %%% last
  the previous section,
the ``spread out'' nature of a path distinguishes the problem from most of the hypergraph Tur\'an problems that
have been studied.  For $k=2$, the problem was solved
by Erd\H{o}s and Gallai in the following classic theorem.

\begin{theorem} {\bf(Erd\H{o}s-Gallai~\cite{erdos-gallai})}
Let $G$ be a graph on $n$ vertices containing no path of length $\ell$.
Then $e(G)\leq \frac{1}{2}(\ell-1)n$. Equality holds iff $G$ is the disjoint union
of complete graphs on $\ell$ vertices.
\end{theorem}

For $k\geq 3$,  the most general definition of a $k$-uniform path %%% of length $\ell$
   is that of a Berge path.
A %%% $k$-uniform
  {\it Berge path} of length $\ell$  is a family of distinct sets $\{F_1,\ldots, F_\ell\}$ and $\ell+1$ distinct vertices
  $v_1,\ldots, v_{\ell+1}$ such that for each $1\leq i\leq \ell$,  $F_i$ contains $v_i$ and $v_{i+1}$.
Let $\cB^{(k)}_\ell$ denote the family of $k$-uniform Berge paths of length $\ell$.
Gy\H ori et al. determined $\ex_k(n,\cB^{(k)}_\ell)$ exactly for infinitely many $n$.

\begin{theorem} {\bf(Gy\H ori et al.~\cite{gyori})}
If $\ell>k\geq 2$ then $\ex_k(n,\cB^{(k)}_\ell)\leq \frac{n}{\ell}{\ell\choose k}$. Furthermore, equality is attained
if $\ell$ divides $n$. If $3\leq \ell\leq k$, then $\ex_k(n,\cB^{(k)}_\ell)=\frac{n(\ell-1)}{k+1}$.
Furthermore, here equality is attained if $k+1$ divides $n$.
\end{theorem}
For the $\ell>k$ case, equality is attained by partitioning the $n$ vertices into sets of size $\ell$ and taking
a complete $k$-uniform hypergraph on each of the $\ell$-set. For the $3\leq \ell\leq k$ case, equality is attained by partitioning
the $n$ vertices into sets of size $k+1$ and taking exactly $\ell-1$ of the $k$-sets in each of these $(k+1)$-sets.
   The case $\ell=2$, $\ex_k(n,\cB^{(k)}_2)=\lfloor n/k \rfloor$, is obvious.

A notion that is more restrictive than a Berge path is that of a  loose path.
A  %%%  $k$-uniform
   {\it loose  path} of length $\ell$ is a family of sets $\{F_1,\ldots, F_\ell\}$ such that $F_i\cap F_j\neq \emptyset$
iff $|i-j|=1$. Let $\cP^{(k)}_\ell$ denote the family of $k$-uniform loose paths of length $\ell$.
%%% Mubayi an Verstra\"ete~\cite{MV1} initiated the study of $\ex_k(n,\cP^{(k)}_\ell)$.

\begin{theorem}{\bf(Mubayi-Verstra\"ete~\cite{MV1})} \label{mv}
Let $k,\ell\geq 3$, $t=\fl{(\ell-1)/2}$ and $n\geq (\ell+1)k/2$. Then  $\ex_k(n,\cP^{(k)}_3)={n-1\choose k-1}$.
For $\ell,k>3$, we have $$t{n-1\choose k-1}+O(n^{k-2})\leq \ex_k(n,\cP^{(k)}_\ell)\leq 2t{n-1\choose k-1}+O(n^{k-2}).$$
 \end{theorem}

An even more restrictive notion than that of a loose path is the notion of a linear path. A  %%% $k$-uniform
{\it linear path}  of length $\ell$ is a family of sets $\{F_1,\ldots, F_\ell\}$ such that $|F_i\cap F_{i+1}|=1$ for each $i$
 and $F_i\cap F_j=\emptyset$ whenever $|i-j|>1$.
Let $\pkl$ denote the $k$-uniform linear path of length $\ell$.  It is unique up to isomorphisms.
The determination of $\ex_k(n,\pkl)$ is nontrivial
even for $\ell=2$. This was solved by Frankl~\cite{frankl1}
(see~\cite{KMW} for more on the $k=4$ case).
The case $\ell<k$ was asymptotically determined in~\cite{frankl-furedi}.
As the main result of this paper, we determine $\ex_k(n, \pkl)$ exactly,
for {\bf all} fixed $k,\ell$, where $k\geq 4$, and sufficiently large $n$.

\begin{theorem} \label{main} {\bf(Main result)}
Let $k,t$ be positive integers,  %%% where
    $k\geq 4$. For sufficiently large $n$, we have
$$\ex_k(n,\podd)={n-1\choose k-1}+{n-2\choose k-1}+\ldots+{n-t\choose k-1}.$$
The only extremal family consists of all
the $k$-sets in $[n]$ that meet some fixed set $S$ of $t$ elements. Also,
    %%% for the case $k\geq 4$, $t\geq 0$ and $n> n_{k,t}$ we have
$$\ex(n,\peven)={n-1\choose k-1}+{n-2\choose k-1}+\ldots+{n-t\choose k-1}+{n-t-2\choose k-2}.$$
The only extremal family consists of all
the $k$-sets in $[n]$ that meet some fixed set $S$ of $t$ elements plus all the $k$-sets in $[n]\setminus S$
that contain some two fixed elements. % $u$ and $v$.
\end{theorem}

Our method does not quite work for  the $k=3$ case. We conjecture that a similar result holds for $k=3$.
Using essentially the same method (for $k\geq 4$) and a slight modification of the method (for $k=3$), one
can also determine the Tur\'an numbers of loose paths for all fixed $k\geq 3$ and large $n$.

\begin{theorem} \label{loose-paths} \label{Jiang-Seiver}
Let $k,t$ be positive integers, where $k\geq 3$.    %%% ??? !!! $k\geq3$.
For sufficiently large $n$, we have
$$\ex_k(n,\cP^{(k)}_{2t+1})={n-1\choose k-1}+{n-2\choose k-1}+\ldots+{n-t\choose k-1}.$$
The only extremal family consists of all
the $k$-sets in $[n]$ that meet some fixed set $S$ of $t$ vertices. Also,
  %%%  for $t\geq 0$, $k\geq 4$ and $n> n_{k,t}$ we have
$$\ex(n,\cP^{(k)}_{2t+2})={n-1\choose k-1}+{n-2\choose k-1}+\ldots+{n-t\choose k-1}+1.$$
The only extremal family consists of all the $k$-sets in $[n]$ that meet some fixed set $S$ of $t$ vertices
plus one additional $k$-set that is disjoint from $S$.
\end{theorem}

Since Theorem \ref{loose-paths} is not our main result and for $k\geq 4$ the proof is essentially the same
as that of Theorem \ref{main}, we will not formally prove Theorem \ref{loose-paths}. We will instead just
briefly comment on how to prove Theorem \ref{loose-paths} at the end of Section 5. For details, see \cite{JS}.

%%% For a lower bound on both $\ex_k(n,\podd)$ and $\ex(n,\cP^{(k)}_{2t+1})$, we let $\cF$ be
%%% the family of all $k$-sets in $[n]$ that intersect some fixed set $S$ of $t$ elements. It is easy to
%%% see that $\cF$ has no $t+1$ independent hyperedges and thus contains no linear (or loose) path of length
%%% $2t+1$. For a lower bound on $\ex_k(n,\peven)$, we let $\cF$ be the family of all $k$-sets in $[n]$
%%% that intersect a set $S$ of $t$ elements plus all the $k$-sets in $[n]\setminus S$ that contain some
%%% fixed two elements $u,v$. It is easy to see that $\cF$ contains no linear path of length $2t+2$.
%%% For a lower bound on $\ex_k(n,\cP^{(k)}_{2t+2})$, we let $\cF$ be the family of all $k$-sets in $[n]$
%%% that intersect a set $S$ of $t$ elements plus one additional $k$-set in $[n]\setminus S$.
   It is easy to see that the constructions  described in the above two theorems are indeed
$\pkl$ and $\cP^{(k)}_{\ell}$-free, respectively.
We will show that for large enough $n$
 they are the unique extremal constructions for the respective Tur\'an numbers.

We organize our paper as follows.
In Section 3, we introduce our main tool: the delta-system method and develop some useful facts.
In Section 4, we establish asymptotically tight bounds. In Section 5, we prove the exact bounds,
characterize the extremal families and establish stability results.  %%% on our bounds.
In Section 6, we prove a related result.
In Section 7 we collect  a few problems and remarks.

%%%%%%%%%%%%%%%%%%%%%%%%%%%%%%%%%%%%%%%%%%%%%%%

\section{The delta-system method and homogeneous families}

The {\em delta-system method},  started by  Deza, Erd\H os and Frankl~\cite{DEF} and others,
  is a powerful tool for solving set system problems.
The method is   %%% built upon  a  %%% beautiful
 summarized in a
 structural lemma obtained by the first author~\cite{furedi-1983} (see Lemma \ref{homogeneous} below).
 %%% F\"uredi \cite{furedi-1983} .
 It has been used successfully to obtain a series
 of sharp results on set systems,  most notable in~\cite{frankl-furedi}, and more recently in~\cite{lale-furedi}.

We now introduce a few definitions.
A family of sets $F_1,\ldots, F_s$ are said to form an {\it $s$-star}  or {\it $\Delta$-system} of size $s$ with {\it kernel}
 $A$ if $F_i\cap F_j=A$ for all $1\leq i<j \leq s$.
Sets $F_1,\ldots ,F_s$ are called the {\it petals}  (or {\it members}) of the $\Delta$-system.
 Given a family
$\cF$ of sets and a member $F$ of $\cF$, we define the {\it intersection structure} of $F$ relative to
$\cF$ to be
$$\cI(F,\cF)=\{F\cap F': F'\in \cF, F'\neq F\}.$$
In other words, $\cI(F,\cF)$ consists of all the intersections of $F$ with other members of $\cF$.
As in many $k$-uniform hypergraph problems, it is often convenient to assume the family $\cF$ to be $k$-partite.
A $k$-uniform family $\cF\subseteq {[n]\choose k}$ is {\it $k$-partite} if there exists a partition of the vertex
set $[n]$ into $k$ sets $X_1,\ldots, X_k$, called {\it parts}, such that $\forall F\in \cF$ and $\forall i\in [k]$ we have
$|F\cap X_i|=1$. So, each member of  $\cF$ consists of one vertex
from each part.
We will call $(X_1,\ldots, X_k)$ a (vertex) {\em $k$-partition} of $\cF$.
 Recall that an old result of Erd\H{o}s and Kleitman~\cite{EK} showed that every $k$-uniform family
$\cH$ contains a $k$-partite subfamily $\cH' \subseteq \cH$ of size at least $(k!/k^k)|\cH|$.

Let $\cF\subseteq {[n]\choose k}$ be a $k$-partite family with a  $k$-partition $(X_1,\ldots, X_k)$. Given
any subset $S\subseteq [n]$, its {\it  pattern}, denoted by $\Pi(S)$, is defined as
$$\Pi(S)=\{i: S\cap X\neq \emptyset\} \subseteq [k].$$
In other words, the pattern of $S$ records which parts in the given $k$-partition that $S$ meets.
If $\cL$ is a collection of subsets of $[n]$, then we define
$$\Pi(\cL)=\{\Pi(S): S\in \cL\} \subseteq 2^{[k]}.$$
We will call $\Pi(\cI(F,\cF))$ the {\it intersection pattern} of $F$ relative to $\cF$.
%%% We now describe %%% F\"uredi's lemma.
%%%                          the intersection semmilattice lemma.

\begin{lemma}{\bf (The intersection semilattice lemma~\cite{furedi-1983})} \label{homogeneous}
For any positive integers $s$ and $k$, there exists a positive constant $c(k,s)$ such that
{\bf every}
family $\cF\subseteq {[n]\choose k}$ contains a subfamily $\cF^*\subseteq \cF$ satisfying
\begin{enumerate}
\item $|\cF^*|\geq c(k,s)|\cF|$.
\item $\cF^*$ is $k$-partite, together with a $k$-partition $(X_1,\ldots, X_k)$.
\item There exists a family $\cJ$ of proper subsets of $[k]$ such that $\Pi(\cI(F,\cF^*))=\cJ$ holds for
all $F\in \cF^*$.
\item $\cJ$ is closed under intersection, i.e., for all $A,B\in \cJ$ we have $A\cap B\in \cJ$ as well.
\item Fixing any $F\in \cF^*$, for each $A\in \cI(F,\cF^*)$ there exists an $s$-star in $\cF^*$ containing $F$
    %%% and $F'$ and
    with kernel $A$.
\end{enumerate}
\end{lemma}
Note that for $s\geq k$, item 4 follows from items 3 and  5. For $s<k$, observe that if all items hold for $s=k$, then they certainly
also hold for all $s<k$.
\begin{definition}
{\rm We call family $\cF^*$ that satisfies items (2)-(5) of Lemma~\ref{homogeneous} a {\it $(k,s)$-homogeneous family} with
intersection pattern $\cJ$.   When the context is clear we will drop the $(k,s)$-prefix.
}
\end{definition}

A useful notion in the delta-system method is the notion of a rank of family. Given a family $\cL$ of
subsets of $[k]$, we define the {\it rank} of $\cL$, denoted by $r(\cL)$ as
$$r(\cL)=\min\{|D|: D\subseteq [k], \not\exists B\in \cL, D\subseteq B\}.$$
So, $r(\cL)$ is the cardinality of a smallest set $D$ that ``obstructs'' $\cL$ in the sense that
no member of $\cL$ contains it. We will   apply the rank notion to the intersection pattern $\cJ\subsetneq 2^{[k]}$.
If $\cF$ is a $k$-partite family with a $k$-partition $(X_1,\ldots, X_k)$,
$F\in \cF$ and $D\subseteq [k]$, we will let $F[D]=F\cap(\bigcup_{i\in D} X_i)$. That is, $F[D]$ is the projection of $F$ onto
the parts whose indices are in $D$. Given a family $\cF\subseteq {[n]\choose k}$ and a subset $W\subseteq [n]$,
we define the {\it degree} of $W$ in $\cF$ as
$$\deg_\cF(W)=|\{F: F\in \cF, W\subseteq F\}|.$$

\begin{lemma} \label{own-subset}
Let $k,s$ be positive integers. Let $\cF^*$ be a
$(k,s)$-homogeneous family with intersection pattern $\cJ$. Let
$D\subseteq [k]$. Suppose no member of $\cJ$ contains $D$. Then
$\deg_{\cF^*}(F[D])=1$.
\end{lemma}
\begin{proof}
Suppose that $F'$ is another member of $\cF^*$ besides $F$ that contains $F[D]$.
Then $F\cap F'\supseteq F[D]$. Let $B=\Pi(F\cap F')$.
Then $B\supseteq D$. Since $\cF^*$ is homogeneous with intersection pattern $\cJ$,  $B\in \cJ$.
This contradicts our assumption that  no member of $\cJ$ contains $D$.
\end{proof}

Lemma~\ref{own-subset} immediately implies

\begin{proposition}{\bf (The rank bound)} \label {rank-bound}\enskip
Let $k,s$ be positive integers. Let $\cF^*$ be a
$(k,s)$-homogeneous family on $n$ vertices with intersection
pattern $\cJ$. If $r(\cJ)=p$, then $|\cF^*|\leq {n\choose p}$.
\end{proposition}
\begin{proof}
By definition,  $\exists D\subseteq [k]$ with $|D|=p$ such that no member of $\cJ$ contains $D$.
By Lemma~\ref{own-subset}, for each $F\in \cF^*$, $F[D]$ is a $p$-subset of $F$ that is not contained
   in any other member of $\cF^*$.
Suppose $F_1,\ldots, F_m$ are all the members of $\cF^*$. Then $F_1[D],F_2[D],\ldots, F_m[D]$ are
all distinct $p$-sets, and clearly there can be at most ${n\choose p}$ of them. So, $|\cF^*|=m\leq {n\choose p}$.
\end{proof}

In the spirit of Proposition~\ref{rank-bound}, we will focus on
homogeneous families whose intersection patterns $\cJ$ have  rank
$k-1$ or $k$.  Among rank $k-1$ patterns, we consider two types.

\begin{definition}
Let $\cL$ be a family of proper subsets of $[k]$ that has rank $k-1$. We say that $\cL$ is of type $1$ if
there exists an element $x\in [k]$ such that $[k]\setminus \{x\}\notin \cL$ but $\forall y\in [k], y\neq x,
[k]\setminus \{y\}\in \cL$.
   If $\cL$ has rank $k-1$, but is not of type $1$, then we say that it is of type $2$.
\end{definition}

We now prove some quick facts.

\begin{lemma}  \label{rank-facts}
Let $k\geq 3$  be a positive integer. Let $\cL$ be a family of proper subsets of $[k]$ that is closed under intersection.
\begin{enumerate}
\item If $\cL$ has rank $k$, then it consists of all the proper subsets of $[k]$.
\item If $\cL$ has rank $k-1$ and is of type 1, then for some $i\in [k]$, $\cL$ contains all the proper subsets of
$[k]$ that contain $i$. We will call $i$ the {\it central element}.
\item For $k\geq 4$, if $\cL$
has rank $k-1$ and is  of type $2$ then $\cL$ contains at least two singletons.
\end{enumerate}
\end{lemma}
\begin{proof}
First, assume that $\cL$ has rank $k$. By the definition of rank,
every $(k-1)$-subset of $[k]$ belongs to $\cL$. Since $\cL$ is closed under intersection,
every proper subset of $[k]$ is in $\cL$.

Next, suppose that $\cL$ has rank $k-1$ and is of type $1$. By definition, there exists $i\in [k]$
such that $[k]\setminus \{i\}\notin \cL$ but $\forall j\in [k], j\neq i,  [k]\setminus \{j\}\in \cL$. Since $\cL$ is closed under
intersection it   %%% is easy to see that $\cL$
   contains all the proper subsets of $[k]$ that contain $i$.

Finally, assume that $\cL$ has rank $k-1$ and is of type $2$.  By definition, there are some $t\ge 2$  different $(k-1)$-subsets
of $[k]$ that obstruct $\cL$.
Without loss of generality, we may assume that
$\forall i=1,\ldots, t, [k]\setminus \{i\} \notin \cL$ and $\forall i=t+1,\ldots, k, [k]\setminus \{i\} \in \cL$.

\medskip

{\bf Claim.}  $\forall i,j\leq t, i\neq j$, we have $[k]\setminus \{i,j\}\in \cL$.

{\it Proof of Claim.} Otherwise suppose for some $i,j\leq t, i\neq j$,  $D=[k]\setminus \{i,j\}\notin \cL$.
%%% Then s
    Since $r(\cL)=k-1>k-2$, there must be some member of $\cL$ that contains $D$. However, the
only possible members of $\cL$ that could contain $D$ are $[k]\setminus \{i\}$ and $[k]\setminus \{j\}$,
neither of which is in $\cL$, a contradiction.

\medskip

By our discussions above, we know $\forall i=t+1,\ldots, k, [k]\setminus \{i\}\in \cL$ and
    $\forall i,j\in [t], [k]\setminus \{i,j\}\in \cL$ and  $\cL$ is closed under intersection.
If $t\geq 3$, then %%% one can check that
    $\{i\}\in \cL$ for each $i\in [t]$.
If $t=2$, then   %%% one can check that
     $\{i\}\in \cL$ for each $i\in \{3,\ldots, k\}$.
So, in particular, if $k\geq 4$ then $\cL$ contains at least two singletons.
\end{proof}

Lemma~\ref{rank-facts} immediately yields

\begin{corollary}\label{singleton}
Let $k,s$ be positive integers, where $s\geq k\geq 4$.
Let $\cF^*$ be a $(k,s)$-homogeneous family with intersection
pattern $\cJ$.
Suppose $\cJ$ has rank $k$ or has rank $k-1$ and is of type $2$. Let $F\in \cF^*$.
Then there exist at least two distinct vertices $u,v\in F$ such that $\{u\}$ is the kernel of some $s$-star in $\cF^*$
and $\{v\}$ is the kernel of some $s$-star in $\cF^*$.
\end{corollary}
\begin{proof}
By Lemma~\ref{rank-facts}(3), there exist $i,j\in [k]$ such that $\{i\}\in
\cJ$ and $\{j\}\in \cJ$. Let $u=F[\{i\}]$ and $v=F[\{j\}]$. Since $\cF^*$ is homogeneous, $\{u\}=F[\{i\}]\in
\cI(F,\cF^*)$ and $\{v\}=F[\{j\}]\in \cI(F,\cF^*)$. By Lemma~\ref{homogeneous}(5), each of $\{u\}$
and $\{v\}$  is the kernel of some $s$-star in $\cF^*$.
\end{proof}

A hypergraph (set system) $\cH$  is {\it linear} if every two members of $\cH$
intersect in at most one vertex.
Given a graph $H$, the {\it $k$-blowup}, denoted by $[H]^{(k)}$   (or $H^{(k)}$ for short), is the $k$-uniform hypergraph obtained
from $H$ by replacing each edge $xy$ in $H$ with a $k$-set $E_{xy}$ that consists of $x,y$ and $k-2$ new
vertices such that for distinct edges $xy, x'y'$, $(E_{xy}-\{x,y\})\cap (E_{x'y'}-\{x',y'\})=\emptyset$.
If $H$ has $p$ vertices and $q$ edges, then $H^{(k)}$ has $p+q(k-2)$ vertices and $q$ hyperedges.
The resulting $H^{(k)}$ is a $k$-uniform linear hypergraph whose vertex set contains the vertex set of $H$.
We call $H$ the {\it skeleton} of $H^{(k)}$.

%%% For the purpose of this paper,  we
   We adopt the convention that $P_\ell$ denotes a path with $\ell$ edges (and $\ell+1$ vertices).
Then $[P_\ell]^{(k)}$ is a $k$-uniform linear path of length $\ell$. Throughout the paper,  we %%% will
    denote this hypergraph by $\mathbb{P}_\ell^{(k)}$.

\begin{theorem} \label{tree}
Let $k,s,q$ be positive integers where $k\geq 4$ and $s\geq kq$.
Let $T$ be an $q$-edge tree. Let $\cF^*$ be a $(k,s)$-homogeneous family
with intersection pattern $\cJ$. If $\cJ$ has rank $k$ or has rank $k-1$ and is of type $2$,
then $T^{(k)}\subseteq \cF^*$.
\end{theorem}
\begin{proof}
For convenience, if $\{x\}$ is the kernel of an $s$-star in $\cF^*$ we call $x$ a {\it kernel vertex} in $\cF^*$.
We use induction on $q$ to find a copy of $T^{(k)}$ in $\cF^*$ in which each vertex of $V(T)$
is mapped to a kernel vertex in $\cF^*$.
For the basis step, let $q=1$. So $T$ consists of a single edge $xy$. We take any member $F\in \cF^*$.
By Corollary~\ref{singleton}, there exist $u,v\in F$ that are kernel vertices in $\cF^*$.
Now, $F$ is a copy of $T^{(k)}$. Furthermore, by mapping $x$ to $u$ and $y$ to $v$, we fulfill the additional
requirement that each vertex in $V(T)$ is mapped to a kernel vertex in $\cF^*$.
For the induction step, let $q\geq 2$. Let $v$ be a leaf of $T$ and
$u$ its unique neighbor in $T$. Let $T_1=T-v$. By induction hypothesis, $\cF^*$ contains a copy $L$
of $[T_1]^{(k)}$ in which each vertex of $V(T_1)$ is mapped to a kernel vertex in $\cF^*$.
Suppose $u$ is mapped to $u'$. Then $\{u'\}$ is the kernel of an $s$-star $S$
in $\cF^*$. Suppose
$F_1,\ldots, F_s$ are the petals of $S$. Since $F_1\setminus \{u'\},\ldots, F_s\setminus \{u'\}$ are
pairwise disjoint and $s\geq kq>|L|$, for some $j$, $F_j\setminus \{u'\}$
is disjoint from $L\setminus \{u'\}$. Now $L\cup F_j$ forms
a copy of $T^{(k)}$. Furthermore, by Corollary~\ref{singleton}, $F_j$ contains some $v'$ other than $u'$ that is
a kernel vertex in $\cF^*$. By mapping $v$ to $v'$, we maintain the condition that each vertex of $V(T)$ is mapped to
a kernel vertex in $\cF^*$. This completes  the proof.
\end{proof}

\begin{theorem}\label{partition}
Let $k,l,s$ be positive integers with $k\geq 4$ and $s\geq kl$.
Let $\cF\subseteq {[n]\choose k}$. Suppose $\pkl\not\subseteq \cF$. Then
$\cF$ can be partitioned into
subfamilies $\cG_1,\ldots, \cG_m, \cF_0$ such that $\forall i\in [m]$, $\cG_i$ is $(k,s)$-homogeneous with
   intersection pattern $\cJ_i$ which has rank $k-1$ and
type $1$, and $|\cF_0|\leq \frac{1}{c(k,s)}{n\choose k-2}$.
\end{theorem}

\begin{proof}
First we apply Lemma~\ref{homogeneous} to $\cF$ to get a $(k,s)$-homogeneous subfamily $\cG_1$ with
intersection pattern $\cJ_1$ such that $|\cG_1|\geq c(k,s) |\cF|$. Then we apply Lemma~\ref{homogeneous}
again to $\cF-\cG_1$ to get a homogeneous subfamily $\cG_2$ with intersection pattern $\cJ_2$ such
that $|\cG_2|\geq c(k,s)(|\cF|-|\cG_1)$. We continue like this. Let $m$ be the smallest nonnegative integer
such that $\cJ_{m+1}$ has rank $k-2$ or less.  Let $\cF_0=\cF-(\bigcup_{i=1}^m \cG_i)$.
By our procedure, $|\cG_{m+1}|\geq c(k,s)|\cF_0|$.
Since $\cJ_{m+1}$ has rank at most $k-2$,
by Lemma~\ref{rank-bound}, $|\cG_{m+1}|\leq {n\choose k-2}$ and hence
 $|\cF_0|\leq \frac{1}{c(k,s)}{n\choose k-2}$.

By our assumption, $\cJ_1,\ldots, \cJ_m$ all have rank at least $k-1$.
If for some $i$,  either $\cJ$ has rank $k$ or has rank $k-1$ and is of type $2$, then by Theorem~\ref{tree},
$\pkl\subseteq \cG_i\subseteq \cF$, contradicting our assumption that $\pkl\not\subseteq \cF$. So
for each $i\in [m]$, $\cJ_i$ is of type $1$.
\end{proof}

For the remaining sections, we will refer to the partition given in Theorem~\ref{partition} as a {\it canonical
partition} of $\cF$.

%%%%%%%%%%%%%%%%%%%%%%%%%%%%%%%%%%%%%%%%%%%%%%%%%%%%

\section{Kernel graphs and asymptotic bounds}
In this section, we introduce some auxiliary graphs associated with the given family $\cF\subseteq{[n]\choose k}$. Using
these we can quickly establish asymptotically tight bounds on $\ex_k(n,\podd)$ and $\ex_k(n,\peven)$. Some of the
definitions and lemmas in this section may be of independent interests.
Given a family $\cF\subseteq {[n]\choose k}$ and
a subset $W\subseteq [n]$, we define the {\it kernel degree} of $W$, denoted by $\deg^*_\cF(W)$, as
$$\deg^*_\cF(W)=\max \{s: \exists \mbox{ an $s$-star with kernel $W$  in } \cF\}.$$
Note that the {\it kernel degree} of $W$ is a much stronger notion than the  degree of $W$.

\begin{definition}\label{kernel-definition}
{\rm Given a family $\cF\subseteq {[n]\choose k}$, the {\it kernel-graph} with {\it threshold } $s$ is a graph $L$ on $[n]$
such that $\forall x,y\in [n]$,  $xy\in E(L)$ iff $\deg^*_{\cF}(\{x,y\})\geq s$.  }
\end{definition}

\begin{lemma} \label{kernel-graph}
Let $H$ be a graph with $q$ edges. Let $s=kq$.
Let $\cF\subseteq {n]\choose k}$ .  Let $L$ be the % Kernel
 kernel graph of $\cF$ with threshold $s$.
If $H\subseteq L$, then $\cF$ contains a copy of $H^{(k)}$ whose skeleton is $H$.
\end{lemma}
\begin{proof}
Let $e_1,\ldots e_q$ be the edges of $H$.  For each $i$, suppose the two endpoints of $e_i$ are
$x_i$ and $y_i$.
We will replace each $e_i$ with a member $E_i$ of $\cF$ that contain $x_i,y_i$
such that $E_1\setminus\{x_1,y_1\}, \ldots, E_q\setminus \{x_q,y_q\}$ are pairwise disjoint.
Since $x_1y_1\in E(L)$, $\deg^*_\cF(\{x_1,y_1\})\geq s$.
Let $E_1$ be any member of $\cF$ that contains $x_1$ and $y_1$ and avoids all $x_2, y_2, \dots, x_q, y_q$.
In general, suppose we have found $E_1,E_2,\ldots, E_{i-1}$.
Since $x_iy_i\in E(L)$, $\deg^*_{\cF}(\{x_i, y_i\})\geq s$  %%% . So,
 there exists an $s$-star $S$ in $\cF$ with kernel $\{x_i,y_i\}$. Let $F_1,\ldots,
F_s$ denote  the petals of $S$. Since $F_1\setminus \{x_i,y_i\}, \ldots, F_s\setminus \{x_i,y_i\}$
are pairwise disjoint and $|(\bigcup_{j=1}^{i-1} E_j)\setminus \{x_i,y_i\}|<kq=s$,  %%% . For some
   there exists an $h\in[s]$ such that $F_h\setminus \{x_i,y_i\}$ is disjoint from $(\bigcup_{j=1}^{i-1} E_j)\setminus \{x_i,y_i\}$.
We can let $E_i=F_h$. We can continue till we find $E_1,\ldots, E_q$ in $\cF$ that meet the requirements. The system
$\{E_1,\ldots, E_q\}$ forms a copy of $H^{(k)}$ whose skeleton is $H$.
\end{proof}

Suppose $\cF$ can be decomposed into $\cF_1,\ldots, \cF_m$, where for each $i\in [m]$,
$\cF_i$ is homogeneous with intersection pattern $\cJ_i$, where $\cJ_i$ has rank $k-1$ and is of type $1$.
We define the {\it $(k,s)$-homogeneous kernel graph} of $\cF$ as follows. Fix any $F\in \cF$. Suppose
$F\in \cF_p$. By Lemma~\ref{rank-facts}, $\cJ_p$ has a central element $i$ such that all proper subsets of $[k]$
containing $i$ are members of $\cJ_p$. In particular, $\{i\}\in \cJ_p$ and $\{i,i'\}\in \cJ_p$ for each $i'\in [k]\setminus \{i\}$.
So $F[\{i\}]\in \cI(F,\cF_p)$ and $F[\{i,i'\}]\in \cI(F,\cF_p)$ for each $i'\in [k] \setminus \{i\}$. We denote $F[\{i\}]$ by
$c(F)$ and call it the {\it central element} of $F$. Thus, we have  $c(F)\in \cI(F,\cF_p)$
and $\{c(F),y\}\in \cI(F,\cF_p)$ for each $y\in F\setminus \{c(F)\}$.
 Note that although $c(F)$ is uniquely determined, it is possible that
 $c(F')\in F$ for some $F\neq F'\in \cF_j$.
Since $\cF_p$ is $(k,s)$-homogeneous, we have $\deg^*_{\cF_p}(\{c(F)\})\geq s$ and
$\deg^*_{\cF_p}(\{c(F),y\})\geq s$ for each $y\in F\setminus \{c(F)\}$. In particular,
this implies that
$$\deg^*_{\cF}(\{c(F)\})\geq s \,\mbox{ and }\, \deg^*_{\cF}(\{c(F),y\})\geq s  \,\,\mbox{ for }\,    \forall y\in F\setminus \{c(F)\} .$$
We define the {\it $(k,s)$-homogeneous kernel} graph $H$ of $\cF$ to be a directed multi-graph on $[n]$ whose edges consist of all the ordered pairs $(c(F),y)$ over all $F\in \cF$ and $y\in F\setminus c(F)$.
    Note that $H$ has edge-multiplicity at most two.
Furthermore, we mark $c(F)$ for each $F\in \cF$.  Let $H'$ denote the underlying simple undirected graph of $H$.
Note that $H'$ is a subgraph of the kernel graph of $\cF$ with threshold $s$. Also, note that at least one
of the two endpoints of each edge of $H'$ is marked.

\begin{lemma} \label{kernel-bound}
Let $\cF\subseteq {[n]\choose k}$. Suppose that $\cF$ can be partitioned into subfamilies
$\cF_1,\ldots, \cF_m$ such that
for each $i=1,\ldots, m$,  $\cF_i$ is $(k,s)$-homogeneous with intersection
pattern $\cJ_i$ that has rank $k-1$ and is of type $1$.
Let $H$ be the $(k,s)$-homogeneous kernel graph of $\cF$.
Let $H'$ be the underlying undirected simple graph of $H$.
We have $$|\cF|\leq \frac{e(H')}{k-1}{n-2\choose k-2}.$$
\end{lemma}
\begin{proof}
Consider the number $q$ of pairs $(\{x, y\}, F)$ where  $F\in \cF, x,y\in F$ and   %%%  either $xy\in E(H)$ or $yx\in E(H)$.
    $\{x,y\}\in E(H')$.
Each $F\in \cF$ contributes exactly $k-1$ to $q$. On the other hand, for each unordered pair $x,y$
trivially there are at most ${n-2\choose k-2}$ members of $\cF$ that contain $x,y$.  So, each $xy\in E(H')$
contributes at most ${n-2\choose k-2}$ to $q$. So, we have $(k-1)|\cF|=q\leq e(H'){n-2\choose k-2}$.
   %%% , which yields $|\cF|\leq \frac{e(H')}{k-1}{n-2\choose k-2}$.
\end{proof}

Now we are ready to establish asymptotic tight bounds on $\ex_k(n,\podd)$ and $\ex_k(n,\peven)$.
We need the following classical result %well-known fact
 concerning the circumference of a graph.   %%% , also due to Erd\H{o}s and Gallai \cite{erdos-gallai}.

\begin{lemma}\label{circumference} {\bf (Erd\H{o}s and Gallai~\cite{erdos-gallai})}
If $G$ is an $n$-vertex graph that contains no cycle of length at least $c$, where $c\geq 3$, then
$e(G)\leq \frac{1}{2}(c-1)(n-1)$.
\end{lemma}

\begin{theorem} \label{peven-general}
Let $k,t$ be positive integers, where $k\geq 4$.
We have $$\ex_k(n,\podd)\leq \ex_k(n,\peven)\leq t{n-1\choose k-1}+O(n^{k-2}).$$
\end{theorem}

\begin{proof}
Let $\cF\subseteq{[n]\choose k}$ be family that contains no copy of $\peven$. Let $s=k(2t+2)$.
By Theorem~\ref{partition}, there exists a partition of $\cF$ into
$\cG_1,\ldots, \cG_m,\cF_0$, where $|\cF_0|\leq \frac{1}{c(k,s)}{n\choose k-2}$ and
for each $i\in [m]$ $\cG_i$ is $(k,s)$-homogeneous with intersection
pattern $\cJ_i$ that has rank $k-1$ and is of type $1$.
Let $\cF'=\cG_1\cup\ldots \cG_m$. Let $H$ be the $(k,s)$-kernel graph of $\cF'$ and
$H'$ the underlying undirected simple graph of $H$.

\medskip

{\bf Claim 1.}  $H'$ has circumference at most $2t$.

\medskip

{\it Proof of Claim 1.} Otherwise suppose $H'$ contains a cycle $C$ of length at least $2t+1$.
Recall that in each edge of $H'$, at leat one endpoint is marked.  If $C$ has length at least
$2t+2$, then we can find a path of length $2t+1$ on $C$ with one of the endpoints being marked.
If $C$ has length $2t+1$ (which is odd) then we can find two consecutive
vertices on $C$ that are marked, in which case we can find a path of length $2t$ both of whose
endpoints are marked.

In the former case, suppose $x_1x_2\ldots x_{2t+2}$ is a path of length $2t+1$ on $C$ where $x_1$ is marked.
By Lemma~\ref{kernel-graph}, $\cF$ contains a copy $\cP$ of $\podd$ whose skeleton is $x_1x_2\ldots x_{2t+2}$.
Since $x_1$ is marked, $\deg^*_{\cF'}(\{x_1\})\geq s=k(2t+2)$. Let $F_1,\ldots, F_s$ be the petals of
an $s$-star in $\cF$ with kernel $\{x_1\}$. Since $\cP$ has fewer than $k(2t+1)$ vertices and $F_1\setminus\{x_1\},
\ldots, F_s\setminus \{x_1\}$ are pairwise disjoint, for some $h\in [s]$, $F_h\setminus \{x_1\}$ is disjoint
from $\cP$. We can add $F_h$ to $\cP$ to form a copy of $\peven$, contradicting the assumption that $\cF$
contains no $\peven$.

In the latter case, suppose $x_1x_2\ldots x_{2t+1}$ is a path of length $2t$ on $C$ where both
$x_1$ and $x_{2t+1}$ are marked. By Lemma~\ref{kernel-graph}, $\cF$ contains a copy $\cP$
of $\mathbb{P}^{(k)}_{2t}$ whose skeleton is $x_1x_2\ldots x_{2t+1}$. Using that
$\deg^*_{\cF'}(x_1)\geq s$ and $\deg^*_{\cF'}(x_{2t+1})\geq s$, we can extend $\cP$ into a copy
of $\peven$, a contradiction. \qed

\medskip

Since $H'$ has circumference at most $2t$, by Lemma~\ref{circumference} (with $c=2t+1$), we have
$e(H')\leq t(n-1)$. By Lemma~\ref{kernel-graph}, $|\cF'|\leq \frac{t(n-1)}{k-1}{n-2\choose k-2}=t{n-1\choose k-1}$.
Therefore, $|\cF|\leq  t{n-1\choose k-1}+\frac{1}{c(k,s)}{n\choose k-2}   $. \end{proof}
   %%% \leq t{n-1\choose k-1}+O(n^{k-2})$.
Note that if we were to just prove  $\ex_k(n,\podd)\leq t{n-1\choose k-1}+O(n^{k-2})$,
it would have sufficed to just use the Erd\H{o}s-Gallai  theorem on
$\ex(n, P_{2t+1})$ to get $e(H')\leq tn$, from which the bound follows.
%%% \end{proof}

To close this section, we observe that following the arguments in \cite{gyori},  by iteratively
removing vertices of degree at most $(k-1)(\ell-1){n-2 \choose k-2}$ one can prove by induction
that the following bound holds for every $n$:
\begin{equation}\label{eq:all_n}
   \ex(n, \pkl)\leq (k-1)(\ell-1) {n-1 \choose k-1}.
    \end{equation}

Even though this is a weaker bound than Theorem \ref{peven-general} for large $n$, it holds for every $n$.
We will use this  bound in certain estimates in the next section.

%%%%%%%%%%%%%%%%%%%%%%%%

\section{Proof of Theorem~\ref{main} and the stability of the bounds}

In this section, we determine the exact value of $\ex_k(n,\pkl)$ for large $n$.
For convenience, we let $f(n,k,t)={n-1\choose k-1}+{n-2\choose k-1}+\ldots + {n-t\choose k-1}$
and $g(n,k,t)={n-1\choose k-1}+{n-2\choose k-1}+\ldots+ {n-t\choose k-1}+{n-t-2\choose k-2}$.
Here, $f(n,k,t)$ is the number of $k$-sets in $[n]$ that meet a fixed set $S$ of $t$ elements of $[n]$
and $g(n,k,t)$ is $f(n,k,t)$ plus the number of $k$-sets in $[n]\setminus S$ that contain some fixed set of
two elements.
We wish to show that for fixed $k,t$, where $k\geq 4$, $\ex_k(n,\podd)=f(n,k,t)$
and $\ex_k(n,\peven)=g(n,k,t)$.  We already established the lower bounds
in the introduction.  Note that $f(n,k,t)\geq t{n-1\choose k-1}
-c_1n^{k-2}$ and $g(n,k,t)\geq t{n-1\choose k-1}-c_2n^{k-2}$ for some constants
$c_1, c_2$ depending on $k,t$. Let $\cF\subseteq{[n]\choose k}$ be a family that contains
no copy of $\peven$. We may assume that there exists  a constant $c_3$, depending on $k$ and $t$, such
that $|\cF|\geq t{n-1\choose k-1}-c_3n^{k-2}$, since otherwise $|\cF|\leq f(n,k,t)$ and
$|\cF|\leq g(n,k,t)$ already hold. As a key step, we first show that $\cF$ must already have a structure very
similar to the extremal construction.

Let $s=k(2t+2)$. Let $\cG_1,\ldots, \cG_m,\cF_0$ be a canonical partition of $\cF$, where
 $|\cF_0|\leq \frac{1}{c(k,s)}{n\choose k-2}$ and  for each $i\in [m]$, $\cG_i$ is $(k,s)$-homogeneous
with intersection pattern $\cJ_i$ that has rank $k-1$ and is of type $1$. Let $\cF'=\bigcup_{i=1}^m \cG_i$.
Let $H$ be the $(k,s)$-kernel graph of $\cF'$.  Recall that $H$ is a directed multigraph with edge-multiplicity at most $2$.
Let $H'$ denote the underlying undirected simple graph of $H$.
For each $x\in V(H)$, let $d^+(x)$ denote the out-degree of $x$ in $H$.
By Claim 1 of Theorem~\ref{peven-general}, $H'$ has circumference at most $2t$ and so
$e(H')\leq t(n-1)<tn$ and  $e(H)=\sum_{x\in V(H)} d^+(x)\leq 2e(H')< 2tn$.
  Let $D=n^{1-\frac{3/2}{k-1}}$.   %%%% Let $D=n^{1-\frac{1+\beta}{k-1}}$.
Define $$A=\{x\in V(H): d^+(x)\leq D\} , \quad
B= \{x\in V(H): d^+(x) > D\}.$$

Let $\cF_A$ denote the set of members $F$ of $\cF'$ whose central element $c(F)$ lies in $A$.
By our definition of $H$,  we have  %%%it is easy to see that
$$|\cF_A|\leq |A|\cdot {D\choose k-1}<n\cdot D^{k-1}=n^{k-\frac{3}{2}}.$$
Since $\sum_{x\in V(H)} d^+(x)<2tn$, we have $|B|<2tn/D=2tn^{\frac{3/2}{k-1}}$.
The subgraph of $H'$ induced by $B$, denoted by $H'[B]$, also has circumference at most
$2t$ and thus $e(H'[B])< t|B|<2t^2n^{\frac{3/2}{k-1}}$.
Let $\cF_B$ denote the set of members of $\cF'$ that contain edges of $H'[B]$. We have
$$|\cF_B|\leq e(H'[B]){n-2\choose k-2}<2t^2n^{k-2+\frac{3/2}{k-1}}\leq 2t^2n^{k-\frac{3}{2}}.$$

Let $\wt{\cF}=\cF'\setminus (\cF_A\cup \cF_B)=\cF\setminus(\cF_0\cup\cF_A\cup \cF_B)$.
By our discussions above, $$|\wt{\cF}|\geq t{n-1\choose k-1}-O(n^{k-\frac{3}{2}}).$$
By our definition of $\cF_A$ and $\cF_B$, we have
    $$\wt{\cF}=\{F\in \cF': c(F)\in B, |F\cap B|=1\}.$$
Let $\wt{H}$ denote the subgraph of $H$ consisting of all edges going from $B$ to $A$.
Suppose $B=\{x_1,\ldots, x_p\}$.
For each $i\in [p]$, let $d_i= d^+_{\wt{H}}(x_i)$.
Based on %%% our
   the definition of $\wt{\cF}$ %%%   , it is easy to see that
  we have
$$|\wt{\cF}|\leq \sum_{i=1}^p {d_i\choose k-1}.$$
Thus,
\begin{equation} \label{degree-lower}
 \sum_{i=1}^p {d_i\choose k-1}\geq t{n\choose k-1}-O(n^{k-\frac{3}{2}}).
\end{equation}
Since $d_i<n$ we get $p\geq t$.
On the other hand, we have
\begin{equation}\label{degree-upper}
\sum_{i=1}^p d_i=e(\wt{H})\leq e(H')< tn.
\end{equation}
Without loss of generality, we may assume that $d_1\geq d_2\geq \ldots \geq d_p$.

\medskip

{\bf Claim 2.} We have $d_1,\ldots, d_t\geq n-O(n^{\frac{1}{2}})$.

\medskip

{\it Proof of Claim 2.} Let $b=\sum_{i=1}^p d_i$ and $y=d_t$.
Using~(\ref{degree-lower})  and convexity we get
   \begin{eqnarray*}
   t{n\choose k-1}-O(n^{k-\frac{3}{2}})&\leq& \sum_{i=1}^p {d_i\choose k-1}\leq \sum_{i=1}^{t-1} {d_i\choose k-1}
      +\frac{b-\sum_{i=1}^{t-1} d_i}{y}{y\choose k-1}\\
   &\leq& (t-1){n\choose k-1}+\frac{tn-(t-1)n}{y}{y\choose k-1}.
   \end{eqnarray*}
Thus, ${y\choose k-1}\frac{n}{y}\geq {n\choose k-1}-O(n^{k-\frac{3}{2}})$.  A standard %%% estimation
  calculation yields $y\geq n-O(n^{\frac{1}{2}})$.
\qed

\medskip
Let us consider again the kernel graph $L$ with threshold $s=(2t+2)k$.
Let $S=\{x_1,\ldots, x_t\}$ and let $W\subseteq [n]-S$ be the maximum set of vertices
  such that $L$ contains all the edges from $S$ to $W$.
Since $H'$ is a subgraph of $L$, Claim 2 implies that

\medskip

{\bf Claim 3.}   $|W|\geq n-O(n^{\frac{1}{2}})$.  \qed

\medskip

Let
$$\cF_S=\{F\in {[n]\choose k}: F\cap S\neq \emptyset\}.$$
We have $|\cF_S|=f(n,k,t)$.
Let $Z= V\setminus (S\cup W)$ and $n_1=|Z|$.
Define
$$\cF_1=\{ F\in \cF: F\subseteq Z\},\quad  \cD=\{F\in  {[n]\choose k}: |F\cap S|= |F\cap Z|=1, F\notin \cF \}.$$
By Claim~3 we have $n_1=O(n^{1/2})$.   By (\ref{eq:all_n})  we have

\begin{equation}\label{F1-upper}
|\cF_1|\leq k(2t+2){n_1-1\choose k-1}.
\end{equation}

Let $z\in Z$.  By the definitions of $W$ and $L$ there exists an $x\in S$ with $\deg_\cF^*(x,z)<s$.
This implies that the $(k-2)$-uniform family
 $$
   \{  F\setminus \{x,z\}: \{ x,z\}\subseteq F\in\cF, \, |F\cap W|=k-2 \}
    $$
  contain no  $s$ pairwise disjoint members, so its size is at most
 $s {|W| \choose k-3}$ by Theorem~\ref{matching}.
Hence $\deg_\cD(x,z)\geq {|W|\choose k-2}-s {|W| \choose k-3} $ and
\begin{equation}\label{D-lower}
   |\cD| \geq |Z|\times \left( {|W|\choose k-2}-s {|W| \choose k-3}\right)\geq \Omega (n_1 \cdot n^{k-2}).
\end{equation}

We are ready to complete the proof of the odd case.

\medskip

{\bf Claim 4.} If $\cF$ contains no copy of $\podd$, then $|\cF|\leq f(n,k,t)$. Furthermore, equality holds only
if $\cF$ consists of all the $k$-sets in $[n]$ that meet $S$.

\medskip

{\it Proof of Claim 4.}  First,  we show that every member of $\cF$ that is disjoint from $S$ is contained in $Z$.
Suppose otherwise. Then there is a member $F$ of $\cF$  that is disjoint from $S$ and intersects $W$.
Let $y_1$ be any element in $F\cap W$.
Since $L$ has all the edges from $S$ to $W$ and $|W|$ is large, one can find a path
  $Q=y_1x_1y_2x_2\ldots y_t x_t y_{t+1}$ of length $2t$ in $L$, where $y_1,\ldots, y_t\in W$,
  such that $Q\cap F=\{ y_1\}$.
Using the fact that for each adjacent pair $u,v$ on $Q$,
   $\deg^*_{\cF}(\{u,v\})\geq s=k(2t+2)$,
  we can extend $F\cup Q$ into a copy of $\podd$, a contradiction.

By the definitions of $\cF_S, \cF_1, \cD$ and our discussion above, we have $\cF \subseteq (\cF_S \setminus \cD) \cup \cF_1$.  By Equations (\ref{F1-upper}) and (\ref{D-lower}) and the fact that $n_1=O(n^{\frac{1}{2}})$ we have
\begin{equation} \label{stable}
f(n,k,t)-|\cF|=|\cF_S|-|\cF|\geq  |\cD|-|\cF_1|\geq
\Omega (n_1\cdot n^{k-2}).
\end{equation}
In particular, we have $|\cF|\leq f(n,k,t).$ Furthermore, equality holds only if $|Z|=n_1=0$ and $\cF=\cF_S$. \qed

Now, we prove the even case.

\medskip

{\bf Claim 5.} If $\cF$ contains no copy of $\peven$ then $|\cF|\leq g(n,k,t)$. Furthermore, equality holds only if $\cF$ consists of all
the $k$-sets in $[n]$ that meet $S$ plus all the $k$-sets in $[n]\setminus S$ that contain two fixed elements.

\medskip

{\it Proof of Claim 5.}
In addition to sets $\cF_S, \cF_1$, and $\cD$, we define
$$\cF_2=\{F\in \cF: F\cap S=\emptyset, |F\cap W|\geq 2\},\quad   \cF_3=\{F\in \cF: F\cap S=\emptyset, |F\cap W|=1\}. $$
Obviously
\begin{equation}\label{eq:sum}
  \cF\subseteq (\cF_S\setminus \cD)\cup \cF_1\cup \cF_2\cup \cF_3.
\end{equation}
Next, we obtain upper bounds on $|\cF_2|$ and $|\cF_3|$.

An $r$-intersecting family is a family of sets in which every two members intersect in at least $r$ elements.
Erd\H{o}s-Ko-Rado~\cite{EKR} showed that for fixed $k$ and large $n$ the unique largest $r$-intersecting family
in $[n]$ is given by the family of all $k$-sets containing a fixed set of $r$ elements (for $n > n(k,r)$).

We claim that $\cF_2$ is a $2$-intersecting family in $[n]\setminus S$.
Otherwise, we can find two members $E_1$ and $E_2$ of $\cF_2$ such that
either $|E_1\cap E_2|=0$ or $|E_1\cap E_2|=1$.
In the former case, %%% like before,
  we can find a path $Q$ of length $2t$ in $L$ using edges between $S$ and $W$
 that meet $E_1$ and $E_2$ each at a single element.
We can then extend $Q\cup E_1\cup E_2$ into a copy of $\peven$, a contradiction.
In the latter case, suppose $E_1\cap E_2=\{y\}$.
Let $w$ be an element in $(E_2\cap W)\setminus \{y\}$.
The element $w$ exists, since $|E_2\cap W|\geq 2$.
We can find a path $Q$ of length $2t$ in $L$ between $S$ and $W$ that meets $E_1\cup E_2$ only in $w$.
Then we can extend $E_1\cup E_2\cup Q$ into a copy of $\peven$, again a contradiction.

We have shown that $\cF_2$ is a $2$-intersecting family in $[n]\setminus S$.
By the Erd\H{o}s-Ko-Rado theorem, (for $n> n_{k,t}$)  we have
\begin{equation}\label{2-intersecting}
|\cF_2|\leq {n-t-2\choose k-2}.
\end{equation}
Furthermore, equality in Equation~(\ref{2-intersecting}) holds only if $\cF_2$ consists of all $k$-sets in
 $[n]\setminus S$ that meet two fixed elements $u,v\in [n]\setminus S$.

Now, consider $\cF_3$.
Let $\wh{\cF}_3=\{F\setminus W: F\in \cF_3\}$. Then $\wh{\cF}_3$ is a collection of $(k-1)$-sets in $Z$.
For a member $C\in \wh{\cF}_3$, define the {\it multiplicity} of $C$ to be the number of different $w\in W$ such that
$C\cup w\in \cF_3$. Let $\wh{\cF}_3'$ denote the set of members of $\wh{\cF}_3$ that have multiplicity $1$
and $\wh{\cF}_3''$ the set of members of $\wh{\cF}_3$ that have multiplicity at least $2$.
Trivially, $|\wh{\cF}_3|\leq {n_1\choose k-1}$. We claim that $\wh{\cF}_3''$ must form an intersecting
family. Otherwise suppose $C_1,C_2$ are two disjoint members of $\wh{\cF}_3''$.
Since $C_1,C_2$ each has multiplicity at least $2$, we can find $w_1,w_2\in W, w_1\neq w_2$, such that
  $E_1=C_1\cup w_1\in \cF_3$ and $E_2=C_2\cup w_2\in \cF_3$.
Now, %%% as in the proof of Claim 4,
  we can find a path $Q$ of length $2t$ in $L$ between
$w_1$ and $w_2$ using edges of $L$ between $S$ and $W$ such that $Q$ intersects $E_1$ only in $w_1$ and
$E_2$ only in $w_2$.
Then we can extend $Q\cup E_1\cup E_2$ into a copy of $\peven$ in $\cF$, a contradiction.
Since $\wh{\cF}_3''$ is an intersecting family in $Z$, by the Erd\H{o}s-Ko-Rado  theorem,
$$ |\wh{\cF}_3''|\leq \min \left\{ {n_1-1\choose  k-2},  {n_1 \choose k-1}\right\} \leq {n_1\choose k-2}.
$$
%%%  Our discussion above shows that $|\wh{\cF}_3|\leq {n_1\choose k-1}+{n_1\choose k-2}={n_1+1\choose k-1}$.
Therefore
\begin{equation}\label{F3}
|\cF_3|   \leq |\wh{\cF}_3'|+|W||\wh{\cF}_3''| \leq
      {n_1\choose k-1}  +n{n_1\choose k-2}
     \leq O\left(  n\cdot n^{k-2}_1 \right).
\end{equation}

By (\ref{eq:sum}), (\ref{F1-upper}), (\ref{D-lower}), (\ref{2-intersecting}), (\ref{F3}), and the fact that
$n_1=O(n^{\frac{1}{2}})$,  we have
\begin{eqnarray}  \label{stable2}
|\cF|&\leq&     |\cF_S|       - |\cD|+| \cF_1|+| \cF_2|+| \cF_3|  \notag  \\
 &\leq&f(n,k,t)+{n-t-2\choose k-2}-\Omega(n_1 \cdot n^{k-2})+O(n \cdot n_1^{k-2}) \notag\\
 &=&g(n,k,t)-\Omega(n_1\cdot n^{k-2}). \label{stable2}
\end{eqnarray}
In particular, we have
$$|\cF|\leq g(n,k,t).$$
Furthermore, equality holds only if $\cF$ consists of all the members of $\cF_S$ plus all the $k$-sets
in $[n]$ that are disjoint from $S$ and contain some two fixed elements $u,v$. \qed

\medskip

With Claim 4 and Claim 5, we have completed the proof of Theorem~\ref{main}.
In addition, Equations~(\ref{stable})
and (\ref{stable2}) imply the following stability result on our bounds.

\begin{theorem} \label{stability}
Let $k,t$ be positive integers, where $k\geq 4$. Let $\varepsilon$ be a small positive real.
There exists a positive real $\delta$ and an integer $n_0$ such that  for all integers $n\geq n_0$
if $\cF\subseteq {[n]\choose k}$ contains no copy of $\peven$ or $\podd$ and $|\cF|\geq (t-\delta){n\choose k-1}$
then there exists a set $S$ of $t$ elements in $[n]$ such that all except at most $\varepsilon {n\choose k-1}$ of the
members of $\cF$ intersect $S$.  \qed
\end{theorem}

To close the section, we briefly remark on how Theorem~\ref{loose-paths} is proved.
For $k\geq 4$ and $\ell$ odd, Theorem~\ref{loose-paths} is implied by Theorem~\ref{main}.
For $k\geq 4$ and $\ell$ even, the proof is
essentially the same %%% up till (\ref{eq:f1}).   We then
except that we replace (\ref{2-intersecting})  with a simpler claim: %%%  that says that
  there is at most one set that is disjoint from $S$ and contained in $W$ and otherwise
  $|\cF_2| = 1+O(n_1^{k-1})$  (since now we are just forbidding a loose path, instead of a linear path).
For the $k=3$ case, the approach is slightly different. We refer interested readers  to~\cite{JS}.

%%%%%%%%%%%%%%%%%%%%%%%%%%%%%%%%%%%%%%%%%%%%%%%%%%%%%%%%

\section{Long linear paths vs. blow-ups of complete bipartite graphs}

In this section, we describe a related result.  First we prove a lemma. In our application
of the lemma, we will choose $m,n$ so that $m=o(n)$.

\begin{lemma} \label{pigeonhole}
Let $b, \ell,q,t$ be positive integers, where $b\geq {\ell\choose t+1}\cdot q +\ell$.
Let $G$ be a bipartite graph with a bipartition $(X,Y)$ where $|X|=m, |Y|=n$. Suppose $e(G)\geq bm+tn$.
Then $G$ contains either a copy of $P_\ell$ or a copy of $K_{t+1,q}$.
\end{lemma}
\begin{proof}
We iteratively remove any vertex in $X$ whose degree becomes less than $b$ and any vertex
in $Y$ whose degree becomes at most $t$. We continue the process until no more vertex
(from either $X$ or $Y$) can be removed. Clearly fewer than $bm+tn$ edges are removed in the process.
So the remaining subgraph $G'$ is non-empty. By design, each vertex on $G'$ in $X$ has degree at least
$b$ and each vertex of $G'$ in $Y$ has degree at least $t+1$. Let $Q$ be a longest path in $G'$.
If $Q$ has length at least $\ell$ then $G$ contains $P_\ell$ and we are done. So we may assume that $Q$ has
length at most $\ell-1$.

Let $v$ be an endpoint of $Q$ and $u$ its unique neighbor on $Q$.
Since $Q$ cannot be extended, we have $d_{G'}(v)\leq \ell-1$,
which implies that $v\in Y$ and hence $u\in X$. Since
$d_{G'}(u)\geq b\geq{\ell\choose t+1}\cdot q+\ell$, $u$ has at
least ${\ell\choose t+1}\cdot q$ neighbors in $G'$ that lie
outside $Q$. None of them has a neighbor outside $Q$ or else we
get a path longer than $Q$, a contradiction. But each of them has
at least $t+1$ neighbors in $G'$ (all of which must lie on $Q$).
By the pigeonhole principle, some $q$ of them are adjacent to the
same set of $t+1$ vertices on $Q$. This gives us a copy of
$K_{t+1,q}$ in $G'\subseteq G$.
\end{proof}

\begin{theorem}\label{path-bipartite}
Let $k,\ell, t,q$ be positive integers where $k\geq 4$.  Let $n$ be
a sufficiently large positive integer depending on $k,\ell$.
There exists a constant $C$ depending on
$k,\ell,t,q$ such that every family $\cF\subseteq {[n]\choose k}$ with $|\cF|\geq t{n\choose k-1}+Cn^{k-2+\frac{2}{k-1}}$
contains either a copy of $\pkl$ or a copy of $[K_{t+1,q}]^{(k)}$.
\end{theorem}
\begin{proof}
The set up of the proof will be similar to that in Section 5.
Let $s=\max\{k\ell, kq(t+1)\}$.
We may assume that $\cF$ contains no copy of $\pkl$ and argue that $\cF$ must contain $[K_{t+1,q}]^{(k)}$.
Let $\cG_1,\ldots, \cG_m,\cF_0$ be a canonical partition of $\cF$, where for each $i\in [m]$,
$\cG_i$ is $(k,s)$-homogeneous with intersection pattern $\cJ_i$ that has rank $k-1$ and is
of type $1$ and $|\cF_0|\leq \frac{1}{c(k,s)}{n\choose k-2}$. Let $\cF'=\bigcup_{i=1}^m \cG_i$.
Let $H$ be the $(k,s)$-homogeneous graph of $\cF'$ and $H'$ the underlying undirected
simple graph of $H$.  Since $\cF'$ doesn't contain
$\pkl$, by Claim~1 of Lemma~\ref{circumference}, $H'$ has circumference less than $\ell$ and
hence $e(H')<\ell n$ and $e(H)\leq 2\ell n$. Therefore $\sum_{x\in V(H)} d^+(x)\leq 2\ell n$.
Let $D=n^{\frac{k-3}{k-1}}$.
Define $$A=\{x\in V(H): d^+(x)\leq D\}  \quad
B= \{x\in V(H): d^+(x) > D\}.$$

Let $\cF_A$ denote the set of members $F$ of $\cF'$ whose central element $c(F)$ lies in $A$.
We have
$$|\cF_A|\leq |A|\cdot {D\choose k-1}<n\cdot D^{k-1}=n^{k-2}.$$
Since $\sum_{x\in V(H)} d^+(x)<2\ell n$, we have $|B|<2\ell n/D=2\ell n^{\frac{2}{k-1}}$.
The subgraph of $H'$ induced by $B$, denoted by $H'[B]$, also has circumference at most
$2\ell$ and thus $e(H'[B])< \ell|B|<2\ell^2n^{\frac{2}{k-1}}$.
Let $\cF_B$ denote the set of members of $\cF'$ that contain edges of $H'[B]$. 
%Equivalently, $\cF_B=\{F\in \cF': c(F)\in B, |F\cap B|\geq 2\}$.  
We have
$$|\cF_B|\leq e(H'[B]){n-2\choose k-2}<2\ell^2n^{k-2+\frac{2}{k-1}}.$$

Let $\wt{\cF}=\cF'\setminus (\cF_A\cup \cF_B)=\cF\setminus(\cF_0\cup\cF_A\cup \cF_B)$.
By our discussions above,  for large enough $C$ we have
\begin{equation} \label{F-tilde}
 |\wt{\cF}|\geq t{n-1\choose k-1}+\frac{C}{2}n^{k-2+\frac{2}{k-1}}.
\end{equation}
By our definition of $\cF_A$ and $\cF_B$, we have
$$\wt{\cF}=\{F\in \cF': c(F)\in B, |F\cap B|=1\}.$$ Let $\wt{H}$ denote the subgraph of $H$
consisting of all edges going from $B$ to $A$.  Note that $\wt{H}$
contains no multiple edges. We will ignore the directions on the
edges of $\wt{H}$ and treat it as an undirected bipartite graph.
Let $b={\ell\choose t+1}\cdot q +\ell$. By Equation
(\ref{F-tilde}) and the proof of Lemma~\ref{kernel-bound}, we have    %%% , for large enough $C$,
$$E(\wt{H})>tn+   %%% \frac{C}{3}
   C n^{\frac{2}{k-1}}\geq t|A|+b|B|,
   $$
By Lemma~\ref{pigeonhole}, $\wt{H}$ contains either a copy of $P_\ell$ or
a copy of $K_{t+1,q}$. Thus, $H\supseteq P_\ell$ or $H\supseteq K_{t+1,q}$. By Lemma~\ref{kernel-graph},
$\cF$ contains either $\pkl$ or $[K_{t+1,q}]^{(k)}$.
\end{proof}

Theorem~\ref{path-bipartite} yields

\begin{corollary} \label{path-bipartite-cor}
Let $k,\ell, t,q$ be positive integers where $k\geq 4$.  We have
$$\ex_k(n,\{\pkl,[K_{t+1,q}]^{(k)}\})\leq t{n-1\choose k-1}+O(n^{k-2+\frac{2}{k-1}}).$$
\end{corollary}

If we set $\ell=2t+2$ and $q=t+2$, then since $K_{t+1,t+2}\supseteq P_{2t+2}$ Corollary~\ref{path-bipartite-cor} yields
\begin{corollary}
Let $k,t$ be positive integers where $k\geq 4$. For sufficiently large $n$ we have
$$\ex_k(n,\podd)\leq \ex_k(n,\peven)\leq t{n-1\choose k-1}+O(n^{k-2+\frac{2}{k-1}}).$$
\end{corollary}
This is almost as good as the bound in Theorem~\ref{peven-general}. However, Corollary~\ref{path-bipartite-cor}
is a more general result, since $\ell$ and $q$ can be an arbitrary constants independent of $t$.
So with essentially the same bound, we get either the blowup of a complete bipartite graph with $t+1$ vertices
on one side or the blow up of an arbitrarily long path.

%%%%%%%%%%%%%%%%%%%%%%%%%%%%%%%%%%%%%%%%%%%%%%%%%%%%%%%%

\section{Remarks and Problems}

\subsection{Triple systems}
Theorem \ref{main} (for linear paths) holds for $k\geq 4$.
Our method does not quite work for  the $k=3$ case.
We {\bf conjecture} that a similar result holds for $k=3$.
On the other hand,  as remarked at the end of Section 5,
Theorem \ref{loose-paths} (for loose paths) does hold for the $k=3$ case using the approach given
in this paper. 

\subsection{Hamilton paths and cycles}
Since the paper by  Katona and Kierstead~\cite{KK}
there is a renewed interest
 concerning paths and (Hamilton) cycles in uniform hypergraphs.
Most of these are Dirac type results
 (large minimum degree implies the existence of the desired substructure)
 like in K\"uhn  and  Osthus~\cite{KO},
 R\"odl, Ruci\'nski,  and Szemer\'edi~\cite{RRS}
 or in Dorbec, Gravier, and S\'ark\"ozy~\cite{DGS}.

\subsection{Long paths}
As the value of $c(k,s)$ in Lemma~\ref{homogeneous} is double exponentially
 small in $k$ and $s$ one can see that our exact results hold for
$$k \ell = O(\log \log n).$$
It would be interesting to close the gap, especially we {\bf conjecture} that
  our result holds for much larger $\ell$, maybe till $k \ell$
is as large as $O(n)$.

\subsection{Linear trees}
A family of sets $F_1, \dots, F_\ell$ is called a \emph{linear}
tree, if  $F_i$ meets $\cup_{j<i} F_j$ in exactly one vertex for all $1< i\leq \ell$.
Any (usual, 2-uniform) tree
 $\TT$ can be blown up in a natural way to a $k$-uniform linear tree $\TT^{(k)}$.
If the minimum number of vertices to cover all edges of  $\bbT$ is
$\tau$, then
 $\ex_k(n, \TT^{(k)})\geq (\tau -1+o(1))  \binom{n-1}{k-1}$.
But one can make a better lower bound. Define
$$\sigma (\TT):=\min \{  |A|+ e(\TT\setminus A): A\subset V(\bbT) \, {\rm independent }\}.$$
We have  ${  \tau \leq \sigma } \leq |V_1|\leq |V_2|$, where
$V_1\cup V_2=V$ is the unique
 two-coloring of $\bbT$.
%%% \noindent
Define
%\centerline{
$$ \cF_0^{(k)}(n,s):=\{ F\in \binom{[n]}{k}: |F\cap \{ 1,2,\dots,  s\} |=1 \}.
 $$
%}
In case of $s<\sigma$ this hypergraph does not contain $\TT^{(k)}$.
\begin{theorem}{\bf \cite{furedi_2011}}
$$
 (\sigma -1) {n-\sigma +1\choose k-1} \leq \ex(n, \TT^{(k)})=  (\sigma -1+o(1)) {n-1\choose k-1}.
  $$
 \end{theorem}
It would be interesting to find asymptotics for the Tur\'an numbers of other linear trees.

\subsection{Kernel graphs and $k$-blowups}

The Kernel graph approach we developed in this paper can potentially be very useful in
attacking other hypergraph Tur\'an problems, particularly the ones concerning
$k$-blow ups of other graphs besides paths. Some related  notions of  expanded graphs were
investigated in earlier papers such as in~\cite{dhruv},~\cite{oleg}, and~\cite{sidorenko}.
The use of appropriately defined auxiliary graphs may ultimately provide a useful approach
for extending extremal results on graphs to hypergraphs
(see~\cite{KM} for example for a successful use of so-called link graphs ).

\subsection{The Erd\H os--S\'os and the Kalai conjecture}\label{traces}

A system of $k$-sets $\TT:=\{ E_1, E_2, \dots, E_q\}$ is called a
{\bf tight} tree
 if for every $2\leq i\leq q$ % the edge $E_i$ adds one new vertex,
 we have $|E_i \setminus \cup_{j< i} E_j|=1$, and there exists an
  $\alpha=\alpha(i)< i$
% $E_\alpha$ preceding $E_i$ (i.e., $\alpha_i < i$)
 such that $|E_\alpha\cap E_i|=k-1$.
The case $k=2$ corresponds to the usual trees in graphs. Let $\TT$
be a $k$-tree on $v$ vertices, and
 let $\ex_k(n,\TT)$ denote the maximum size of a
$k$-family on $n$ elements without  $\TT$.
Consider a $P(n,v-1,k-1)$ packing $P_1$, \dots, $P_m$ on the
vertex set $[n]$
  (i.e., $|P_i|=v-1$ and $|P_i\cap P_j|< k-1$ for $1\leq i<j\leq m$) and replace
  each $P_i$ by a complete $k$-graph.
We obtain a $\bbT$-free hypergraph. Then  R\"odl's~\cite{R}
theorem on almost optimal packings gives
% \begin{equation}\label{eq:k-tree}
$$
   \ex_k(n,\TT)\geq (1-o(1))\frac{{n\choose k-1}}{{v-1\choose k-1}}\times
  {v-1\choose k} = (1-o(1))\frac{v-k}{k}{n \choose k-1}.
  $$ %  \end{equation}
\begin{conjecture}\label{conj2}
{\rm (Erd\H os and S\'os for graphs,
 Kalai 1984 for all $k$, see in~\cite{frankl-furedi})}\newline
$$ \ex_k(n,\TT)\leq \frac{v-k}{k}{n \choose k-1}.$$
  \end{conjecture}

The Erd\H os--S\'os conjecture has been recently proved by a
monumental work of Ajtai, Koml\'os, Simonovits, and
Szemer\'edi~\cite{AKSSzSharp},
 for $v\geq v_0$.

  The Kalai conjecture has been proved for {\bf star-shaped} trees
 in~\cite{frankl-furedi},
i.e., whenever $\TT$ contains a central edge which intersects all
other edges in $k-1$
 vertices.
For $k=2$ these are the diameter 3 trees,  'brooms'.

%%%%%%%%%%%%%%%%%%%%%%%%%%%%%%%%%%%%%%%%%%%%%%%%%%


\begin{thebibliography}{99}



\small

\bibitem{AKSSzSharp}
M. Ajtai, J. Koml\'os, M. Simonovits, E. Szemer\'edi:
The solution of the Erd\H{o}s-S\'os conjecture for large trees, Manuscripts.

\bibitem{bollobas} B. Bollob\'as: {\bf Extremal graph theory},
Academic Press, London,  1978.

\bibitem{DEF} Deza, Erd\H os, Frankl:
Intersection properties of systems of finite sets, \emph{Proc. London Math. Soc. (3)} {\bf 36} (1978),  369--384.


\bibitem{DGS}
Dorbec, P., S. Gravier, G. S\'ark\"ozy:  Monochromatic
Hamiltonian t-tight Berge-cycles in hypergraphs,
\emph{ Journal of Graph Theory} \textbf{59} (2008), 34--44.



\bibitem{erdos-gallai} P. Erd\H{o}s, T. Gallai: On maximal paths
and circuits of graphs, {\em Acta Math. Acad. Sci. Hungar.}
\textbf{10} (1959), 337--356.


\bibitem{erdos-matching} P. Erd\H{o}s: A problem on independent
$r$-tuples, \emph{Ann. Univ. Sci. Budapest} \textbf{8} (1965),
93--95.

\bibitem{EKR} P.  Erd\H{o}s, C. Ko, R. Rado: Intersection theorems for systems of finite sets,
 \emph{Quart. J. Math. Oxford Ser. (2)}  \textbf{12} (1961), 313--320.


\bibitem{EK} P. Erd\H{o}s, D. J. Kleitman: On coloring graphs to maximize the portion of multicolored $k$-edges,
\emph{J. Combin. Th. } \textbf{5} (1968), 164--169.

\bibitem{frankl1}
P. Frankl: On families of finite sets no two of which intersect in a singleton,
\emph{Bull. Austral. Math. Soc.}  {\bf 17} (1977),  125--134.

\bibitem{frankl-furedi} P. Frankl,  Z. F\"uredi: Exact solution of some Tur\'an-type problems,
\emph{J. Combin. Th. Ser. A} \textbf{45} (1987), 226--262.



\bibitem{FK}
P. Frankl,   G. Y. Katona: Extremal k-edge Hamiltonian
hypergraphs, \emph{Discrete Math.} \textbf{308} (2008), 1415--1424.


\bibitem{furedi-1983} Z. F\"uredi:  On finite set-systems whose every intersection is a kernel of a star,
{\em Discrete Math.} \textbf{47} (1983), 129--132.


\bibitem{ZF} Z. F\"uredi: Tur\'an type problems, \emph{Surveys in
Combinatorics}, London Math. Soc. Lecture Note Ser. \textbf{166},
Cambridge Univ. Press, Cambridge, 1991, 253--300.


\bibitem{furedi_2011} Z. F\"uredi:
Linear paths and trees in uniform hypergraphs,  2011.

\bibitem{lale-furedi} Z. F\"uredi, L. \"Ozkahya: Unavoidable subhypergraphs: $a$-clusters,
\emph{J. Combin. Th. Ser. A} \textbf{118}  (2011), 2246--2256 .

\bibitem{gyori} E. Gy\H ori, G.Y. Katona, N. Lemons: Hypergraph extensions of the Erd\H{o}s-Gallai theorem,
\emph{Electronic Notes in Disc. Math.}  \textbf{36} (2010), 655--662.

\bibitem{jiang} T. Jiang, O. Pikhurko, Z. Yilma: Set-systems without a strong simplex,
\emph{SIAM J. Discrete Math.} \textbf{24} (2010), 1038--1045.

\bibitem{JS} T. Jiang, R. Siever: Hypergraph Tur\'an numbers of loose paths, manuscript.

\bibitem{KK}
G. Y. Katona, H. A. Kierstead:  Hamiltonian chains in
hypergraphs,  \emph{J. Graph Theory}  \textbf{30} (1999), 205--212.

\bibitem{keevash} P. Keevash: Hypergraph Turan problems, \emph{Surveys in Combinatorics 2011}, to appear.

\bibitem{KM} P. Keevash, D. Mubayi:  Stability theorems for cancellative hypergraphs, \emph{J. Combin.
Th. Ser B} \textbf{92} (2004), 163--175.

\bibitem{KMW}
P. Keevash, D. Mubayi, R. M.  Wilson:
Set systems with no singleton intersection,
\emph{SIAM J. Discrete Math.}  {\bf 20}  (2006), 1031--1041.


\bibitem{KO}
D. K\"uhn,  D. Osthus:  Loose Hamilton cycles in 3-uniform
hypergraphs of high minimum degree, \emph{J. Combin. Theory Ser. B}
{\bf 96} (2006), 767--821.
% ; MR2274077 (2007h:05115)].

\bibitem{dhruv} D. Mubayi: A hypergraph extension of Tur\'an's theorem, \emph{ J. Combin.
Th. Ser. B} \textbf{96} (2006), 122--134.

\bibitem{MV1} D. Mubayi, J. Verstra\"ete: Minimal paths and cycles in set systems,
\emph{European J. Combin.} \textbf{28} (2007), 1681--1693.

\bibitem{oleg} O. Pikhurko: Exact computation of the hypergraph Tur\'an function
for expanded complete $2$-graph, accepted by \emph{J. Combin. Th. Ser. B}, publication
suspended for an indefinite time, see http://www.math.cmu.edu/pikhurko/Copyright.html.

\bibitem{R}
V. R\"odl:
On a packing and covering problem,
\emph{European J. of Combinatorics} \textbf{6} (1985),  69--78.


\bibitem{RRS}
V. R\"odl, A. Ruci\'nski, E. Szemer\'edi:   An
approximate Dirac-type theorem for k-uniform hypergaphs,
\emph{Combinatorica} \textbf{28} (2008), 229--260.   %%% ES 2006  3-UNIFORM


\bibitem{sidorenko} A. F. Sidorenko: Asymptotic solution for a new class of forbidden $r$-graphs,
\emph{Combinatorica} \textbf{9} (1989), 207--215.

\end{thebibliography}
\end{document}